\documentclass[12pt, notitlepage]{amsart}
\usepackage{latexsym, amsfonts, amsmath, amssymb, amsthm, cite, tabls, paralist}

\newtheorem{theorem}{Theorem}[section]
\newtheorem{lemma}[theorem]{Lemma}
\newtheorem{proposition}[theorem]{Proposition}

\newtheorem{corollary}[theorem]{Corollary}

\theoremstyle{remark}
\newtheorem{remark}[theorem]{Remark}

\theoremstyle{definition}
\newtheorem{definition}[theorem]{Definition}

\theoremstyle{remark}

\numberwithin{equation}{section}

\newcommand{\Epl}{\overset{E}{+}}

\newcommand{\Z}{\mathbb{Z}}

\newcommand{\R}{\mathbb{R}}

\newcommand{\CB}{\mathbf{B}}
\newcommand{\CC}{\mathbf{C}}
\newcommand{\CP}{\mathbf{P}}
\newcommand{\CJ}{\mathcal{J}}

\newcommand{\vol}{\text{vol}}

\usepackage{graphicx}
\usepackage{mathrsfs}

\usepackage{varioref}

\begin{document}

\title{When the sieve works II}
\author{Kaisa Matom\"aki}
\address{Department of Mathematics and Statistics\\ University of Turku \\
20014 Turku \\ Finland}
\email{ksmato@utu.fi}
\thanks{KM was supported by Academy of Finland grants no. 137883 and 138522.}

\author{Xuancheng Shao}
\address{Mathematical Institute\\ Radcliffe Observatory Quarter\\ Woodstock Road\\ Oxford OX2 6GG \\ United Kingdom}
\email{Xuancheng.Shao@maths.ox.ac.uk}
\thanks{XS is supported by a Glasstone Research Fellowship.}

\maketitle

\begin{abstract}
For a set of primes $\mathcal{P}$, let $\Psi(x; \mathcal{P})$ be the number of positive integers $n \leq x$ all of whose prime factors lie in $\mathcal{P}$. In this paper we classify the sets of primes $\mathcal{P}$ such that $\Psi(x; \mathcal{P})$ is within a constant factor of its expected value. This task was recently initiated by Granville, Koukoulopoulos and Matom\"aki \cite{GrKoMa} and their main conjecture is proved in this paper. In particular our main theorem implies that, if not too many large primes are sieved out in the sense that
\[ \sum_{\substack{p \in \mathcal{P} \\ x^{1/v} < p \leq x^{1/u}}} \frac{1}{p} \geq \frac{1 + \varepsilon}{u}, \]
for some $\varepsilon > 0$ and $v \geq u \geq 1$, then
\[ \Psi(x; \mathcal{P}) \gg_{\varepsilon, v} x \prod_{\substack{p \leq x\\ p \notin\mathcal{P}}} \left(1 - \frac{1}{p}\right). \]
\end{abstract}

\section{Introduction}
Let $\mathbb{P}$ be the set of all primes and let $\mathcal{P} \subseteq \mathbb{P}$ be a subset of the primes $\leq x$. We study the most basic sieving problem, wishing to estimate
\[
\Psi(x; \mathcal{P}) := \left|\{ n \leq x \colon p \mid n \implies p \in \mathcal{P}\}\right|.
\]
In other words we sieve the integers in $[1, x]$ by the primes in $\mathcal{P}^c = (\mathbb{P} \cap [1,x]) \setminus \mathcal{P}$. A simple inclusion-exclusion argument suggests that $\Psi(x; \mathcal{P})$ should be approximated by
\[
x \prod_{p \in \mathcal{P}^c} \left(1-\frac{1}{p}\right).
\]
This is always an upper bound, up to a constant, and a lower bound, up to a constant, if $\mathcal{P}$ contains all the primes larger than $x^{1/2-o(1)}$ (see \cite[Theorem 11.13]{Opera} noticing that the sieving limit $\beta = 2$ for $\kappa = 1$). On the other hand there are examples where $\Psi(x; \mathcal{P})$ is much smaller than the expected lower bound. For instance if one fixes $u \geq 1$ and lets $\mathcal{P}$ consist of all the primes up to $x^{1/u}$, then the prediction is about $x/u$ whereas, by an estimate for the number of smooth numbers, we know that $\Psi(x; \mathcal{P}) = \rho(u) x$ with $\rho(u) = u^{-u(1+o(1))} $ as $u \rightarrow \infty$, which is much smaller for large $u$.

The first ones to study what happens if one also sieves out some primes from $[x^{1/2}, x]$ were Granville, Koukoulopoulos and Matom\"aki~\cite{GrKoMa}. They conjectured that the critical issue is what is the largest $y$ such that
\begin{equation}
\label{eq:sum1/pydef}
\sum_{\substack{p \in \mathcal{P} \\ y \leq p \leq x^{1/u}}} \frac{1}{p} \geq \frac{1+\varepsilon}{u}.
\end{equation}
More precisely, they conjectured that when this inequality holds, the sieve works about as expected. On the other hand they gave examples with
\[
\sum_{y \leq p \leq x^{1/u}} \frac{1}{p} = \frac{1-\varepsilon}{u}
\]
such that $\Psi(x; \mathcal{P})$ is much smaller than expected.

Here we continue this study and show that the conjecture indeed holds.
\begin{theorem}
\label{th:MT}
Fix $\varepsilon>0$. If $x$ is large and $\mathcal{P}$ is a subset of the primes $\le x$ for which there are some $1 \leq u \leq v\le \frac{\log x}{1000 \log \log x}$ with
\[
\sum_{\substack{ p\in \mathcal{P} \\ x^{1/v}<p\leq x^{1/u}}} \frac 1p \geq \frac{1+\varepsilon}{u},
\]
then
\[
\frac{\Psi(x;\mathcal{P})}{x} \geq A_v \prod_{p\in \mathcal{P}^c}  \left( 1 -\frac 1p \right),
\]
where $A_v$ is a constant with $A_v = v^{-v(1+o_\varepsilon(1))}$ as $v \to \infty$. If $u$ is fixed, one can take $A_v = v^{-e^{-1/u} v(1+o_\varepsilon(1))}$ as $v \to \infty$.
\end{theorem}

Notice that when $\mathcal{P}$ consists of all the primes $ \leq x^{e^{(1+2\varepsilon)/u}/v}$, the conditions of the theorem are satisfied and an estimate for smooth numbers shows that
\[
\frac{\Psi(x;\mathcal{P})}{x} = v^{-v(e^{-(1+2\varepsilon)/u}+o(1))}
\]
and hence the dependence of the constant $A_v$ on $v$ is close to best possible.

Theorem~\ref{th:MT} establishes the main conjecture of~\cite{GrKoMa}. Granville, Koukoulopoulos and Matom\"aki~\cite[Sections 3--4]{GrKoMa} have reduced a slightly weaker form of the conjecture to an additive combinatorial problem similar to the following hypothesis. We will deduce Theorem~\ref{th:MT} from Hypothesis A in Section~\ref{ssec:HypA->Thm1}.

\theoremstyle{plain}
\newtheorem*{hypa}{Hypothesis A}
\begin{hypa}
Fix $\lambda \in (0,1)$. Let $N \geq v \geq u \geq 1$ be such that $N \geq (100v/\lambda)^2$. If $A$ is a subset of the integers in $(\frac{N}{v},\frac{N}{u}]$ such that
\[
\sum_{\substack{ a\in A  }}  \frac 1a \geq \frac {1+\lambda}{u},
\]
then there exists an integer $k\in[u, v]$ such that
\[
|\{(a_1,\dots,a_k)\in A^k \colon N-k < a_1+\dotsb+a_k \leq N\}| \geq \alpha_v \frac{|A|^{k}}{N},
\]
where $\alpha_v$ is a constant with $\alpha_v = v^{-o_\lambda(v)}$ as $v \to \infty$. If $u$ is fixed and $v \geq 1000u^2/\lambda^2$, one can take $k \leq e^{-1/u} v$. 
\end{hypa}

Furthermore Granville, Koukoulopoulos and Matom\"aki~\cite{GrKoMa} proved (a slight variant of) Hypothesis A for some large constant $\lambda$ and $\alpha_v = v^{-O(v)}$ which implies Theorem~\ref{th:MT} for some large constant $\varepsilon$. Here we will prove Hypothesis A for every $\lambda > 0$ which implies Theorem~\ref{th:MT} for every $\varepsilon > 0$. 

A crucial ingredient is the following result of Bleichenbacher \cite{Bleichenbacher} (see \cite[Section 9]{LePo} for the proof) which may be viewed as a qualitative continuous variant of Hypothesis A.

\theoremstyle{plain}
\newtheorem*{bbthm}{Bleichenbacher's Theorem}

\begin{bbthm}  If $u\geq 1$ and  $T$ is an open subset of $(0,\frac{1}{u})$ for which
\[
\int_{t\in T}\frac {dt}{t} > \frac 1u,
\]
then there exist $t_1,t_2,\dots, t_k\in T$ for which $t_1+t_2+\dotsb+t_k=1$.
\end{bbthm}

Actually we will use the following discrete variant of Bleichenbacher's theorem which is a qualitative variant of Hypothesis A.
\begin{proposition}[Discrete Bleichenbacher]
\label{prop:DiscBlei}
Let $N \geq u \geq 1$ and let $A \subseteq \{1, \dotsc, \lfloor N/u \rfloor\}$ be such that
\[
\sum_{a \in A} \frac{1}{a} > \frac{1}{u} + \frac{1}{\sqrt{N} - 1}.
\]
Then there exists $a_1,\dotsc ,a_k \in A$ such that $N-k < a_1+\dotsb+a_k \leq N$.
\end{proposition}
\begin{proof}
Notice first that the claim follows trivially if there is $a \in A$ such that $a < \sqrt{N}$ since in this case there is $k \geq \sqrt{N}$ such that $N-\sqrt{N} < k a \leq N$. Hence we can assume that $A \subseteq \{\lceil \sqrt{N} \rceil, \lceil \sqrt{N} \rceil + 1, \dotsc, \lfloor N/u \rfloor\}$.

Define $T = \bigcup_{a \in A} \left(\frac{a}{N}, \frac{a+1}{N}\right)$ so that
\[
\begin{split}
\int_{t \in T} \frac{dt}{t} &= \sum_{a \in A} \int_{a/N}^{(a+1)/N} \frac{dt}{t} = \sum_{a \in A} \log\left(1+\frac{1}{a}\right) \geq \sum_{a \in A} \left(\frac{1}{a} - \frac{1}{a^2}\right) \\
&> \frac{1}{u} + \frac{1}{\sqrt{N}-1} - \sum_{a > \lceil\sqrt{N}\rceil} \frac{1}{a^2} > \frac{1}{u}.
\end{split}
\]
Then Bleichenbacher's theorem implies that there are $t_1, \dotsc, t_k \in T$ such that $t_1 + \dotsb + t_k = 1$. For each $j$ there is $a_{i_j}$ such that $t_j \in (a_{i_j}/N, (a_{i_j}+1)/N)$. But then $N-k < a_{i_1} + \dotsb + a_{i_k} < N$.
\end{proof}



The proof of Hypothesis A splits into two cases according to whether much of the set $A$ is contained in $[N/u_0, N/u]$ for some $u_0 = O(1)$ or not. In the first case Hypothesis A follows from an arithmetic removal lemma in a straightforward way, whereas in the second case we develop an analogue of the arithmetic removal lemma with a growing number of variables (see Theorem \ref{thm:pop-rem-largek} below), which could be of independent interest. 

\subsection*{Acknowledgements}  
This work started when both authors were visiting CRM in Montreal during the analytic part of the thematic year in number theory in Fall 2014, whose hospitality is greatly appreciated. Thanks also to Ben Green for helpful discussions.

\section{Some initial reductions}
\label{sec:initred}
\subsection{Deduction of Theorem~\ref{th:MT} from Hypothesis A}
\label{ssec:HypA->Thm1}
As in~\cite{GrKoMa}, we first reduce proving Theorem~\ref{th:MT} to proving a variant of Hypothesis A for the primes called Hypothesis P, and then show that Hypothesis A implies Hypothesis P.

\theoremstyle{plain}
\newtheorem*{hypp}{Hypothesis P}

\begin{hypp}
Fix $\lambda \in (0,1)$. If $x$ is large, $1 \leq u \leq v\le \frac{\log x}{999 \log \log x}$ and $\mathcal{P}$ is a subset of the primes in $(x^{1/v}, x^{1/u}]$ for which 
\[
\sum_{\substack{ p\in \mathcal{P}  }}  \frac 1p \geq \frac {1+\lambda}{u} ,
\]
then there exists an integer $k\in[u, v]$ such that
\[
\left|\left\{(p_1,\dots,p_k)\in \mathcal{P}^k \colon \frac{x}{2} \leq p_1 \dotsm p_k \leq x\right\}\right| \geq \pi_v \cdot \frac{x}{v^k \log x},
\]
where $\pi_v$ is a constant with $\pi_v = v^{-o_\lambda(v)}$ for $v \to \infty$. If $u$ is fixed and $v \geq 1000u^2 /\lambda^2$, one can take $k \leq e^{-1/u} v$. 
\end{hypp}

\begin{proof}[Proof that Hypothesis P implies Theorem \ref{th:MT}] 
We can clearly assume that $\varepsilon < 1/1000$. Let $\mathcal{A}=\mathcal{P}\cap[1,x^{1/v}]$ and $\mathcal{B}=\mathcal{P}\cap(x^{1/v},x^{1/u}]$ so that
\[
\Psi(x;\mathcal{P})\geq \Psi(x;\mathcal{P}\cap [1, x^{1/u}]) \geq \sum_{\substack{ a\leq x^{\varepsilon/(5v)} \\p|a\implies p\in \mathcal{A}}}\Psi(x/a;\mathcal{B}),
\]
since we can write any $n$ composed only of prime factors from $\mathcal{P} \cap [1, x^{1/u}]$ as $n=ab$ where $a$ and $b$ are composed only of prime factors from $\mathcal{A}$ and $\mathcal{B}$, respectively.
For each  $a\le x^{\varepsilon/(5v)}$, we have that
\[
\sum_{\substack{ p\in \mathcal{B} \\(x/a)^{1/v}<p\leq (x/a)^{1/u}}}\frac{1}{p}\ge\sum_{\substack{ p\in \mathcal{P} \\x^{1/v}<p\leq x^{1/u}}}\frac{1}{p} - \sum_{\substack{ p\in \mathbb{P} \\(x/a)^{1/u} <p\leq x^{1/u}}}\frac{1}{p}
\]
If $(x/a)^{1/u} > x^{1/u} - x^{2/(3u)}$, a trivial estimate gives
\[
\sum_{\substack{ p\in \mathbb{P} \\(x/a)^{1/u} <p\leq x^{1/u}}}\frac{1}{p} \leq 
\frac{2x^{2/(3u)}}{x^{1/u}/2}  = 4 x^{-1/3u} \leq \frac{\varepsilon}{2u}.
\] 
Otherwise, Huxley's prime number theorem for short intervals (see e.g. Theorem 10.5 in \cite{IwKo04} and the subsequent discussion) yields, once $x$ is large enough,
\[
\sum_{\substack{ p\in \mathbb{P} \\(x/a)^{1/u} <p\leq x^{1/u}}}\frac{1}{p} \leq 2\log \frac{\log x^{1/u}}{\log (x/a)^{1/u}} \leq - 2 \log \left(1-\frac{\varepsilon}{5v}\right) \leq \frac{\varepsilon}{2u}.
\]

Hence, in any case,
\[
\sum_{\substack{ p\in \mathcal{B} \\(x/a)^{1/v}<p\leq (x/a)^{1/u}}}\frac{1}{p}\ge \frac{1+\varepsilon/2}{u},
\]
and applying Hypothesis P to the set $\mathcal{B}$ yields that there exists $k \in [u, v]$ such that
\[
\Psi(x/a;\mathcal{B}) \geq \pi_v \cdot \frac x{av^k \log x},
\]
where $\pi_v = v^{-o_\varepsilon(v)}$ for $v \to \infty$.
Consequently
\[
\frac{\Psi(x;\mathcal{P})}x\gg \pi_v \frac{1}{v^k \log x}\sum_{\substack{  a \leq x^{\varepsilon/5v} \\p|a\implies p\in \mathcal{A}}}\frac{1}{a}
\gg\frac{\pi_v}{v^k} \prod_{\substack{ p \leq x^{\varepsilon/5v} \\ p\in\mathcal{P}^c }} \left(1-\frac1p\right) \gg \frac{\pi_v}{v^k} \prod_{p\in\mathcal{P}^c } \left(1-\frac1p\right)
\]
by~\cite[Lemma 2.1]{GrKoMa}. This gives the desired lower bound since $k \leq v$, and in case $u$ is fixed, $k \leq e^{-1/u} v$ for large enough $v$.
\end{proof}

\begin{proof}[Proof that Hypothesis A  implies Hypothesis P]
Let $\rho = 1+\left(\frac{\lambda}{1000v}\right)^2$ and $N=\log_{\rho} x-v$. Define, for $j \geq 0$,
\[
A_j = \Bigl\{a \in (N/v , N/u] \colon \sum_{\substack{p \in \mathcal{P} \\ \rho^a \leq p < \rho^{a+1}}} \frac{1}{p} \geq \frac{e^{-j}}{a}\Bigr\}.
\]

Let $J_0 = \log \frac{20 v \log v}{\lambda}$ and let $j_0$ be the smallest integer $j \geq 0$ for which 
\begin{equation}
\label{eq:j0def}
\sum_{a \in A_j} \sum_{\substack{p \in \mathcal{P} \\ \rho^a \leq p < \rho^{a+1}}} \frac{1}{p} \geq \frac{1+\frac{\lambda}{3} + \frac{\lambda}{3} \cdot \frac{j}{J_0}}{u}.
\end{equation}
Notice that, since
\[
\sum_{a \in A_{J_0}} \sum_{\substack{p \in \mathcal{P} \\ \rho^a \leq p < \rho^{a+1}}} \frac{1}{p} \geq \sum_{\substack{p \in \mathcal{P} \\ x^{1/v} < p \leq x^{1/u}}} \frac{1}{p}  - \sum_{N/v < a \leq N/u} \frac{e^{-J_0}}{a} - \sum_{x^{1/u} \rho^{-v/u} \leq p \leq x^{1/u}} \frac{1}{p} \geq \frac{1+2\lambda/3}{u},
\]
necessarily $j_0 \leq J_0$. Write $A = A_{j_0}$. Then, by Huxley's prime number theorem in short intervals,
\begin{equation}
\label{eq:asumA->P}
\sum_{a \in A} \frac{1}{a} \geq \sum_{a \in A} \log\left(1+\frac{1}{a}\right) \geq (1-\lambda/100) \sum_{a \in A_{j_0}} \sum_{\substack{p \in \mathcal{P} \\ \rho^a \leq p < \rho^{a+1}}} \frac{1}{p} \geq \frac{1+\frac{\lambda}{4}}{u}.
\end{equation}
Furthermore, since $j_0$ was chosen to be the smallest integer for which~\eqref{eq:j0def} holds, we get that
\[
\sum_{\substack{a \in A_{j_0}}} \frac{e^{-j_0+1}}{a} \geq \sum_{\substack{a \in A_{j_0} \setminus A_{j_0 - 1}}}  \sum_{\substack{p \in \mathcal{P} \\ \rho^a \leq p < \rho^{a+1}}} \frac{1}{p} \geq \frac{\lambda}{3J_0 u},
\]
so that
\[
|A| \cdot e^{-j_0} \geq  \frac{N}{ev} \cdot \frac{\lambda}{3J_0 u} \gg \frac{\lambda}{v \log v} \cdot \frac{N}{u}.
\]

By~\eqref{eq:asumA->P}, we can apply Hypothesis A to the set $A$ which gives that, for some $k \leq v$ (or in case $u$ is fixed and $v \geq 1000u^2/\lambda^2$, $k \leq e^{-1/u} v$),
\[
|\{(a_1,\dots,a_k)\in A^k \colon  N-k < a_1+\dots+a_k \leq N \}| \geq \alpha_v \frac{|A|^k}{N},
\]
where $\alpha_v = v^{-o_\lambda(v)}$.

Now for each solution to $a_1+\dotsb+a_k \in (N-k, N]$ with
$a_1,\dots,a_k\in A$, consider the primes $p_j\in \mathcal{P}$ with $\rho^{a_j} \leq p_j < \rho^{a_j+1}$. Note that for such primes $p_1 \dotsm p_k \leq \rho^{N+k} \leq x$ and
\[
p_1 \dotsm p_k \geq \rho^{N-k} \geq x \rho^{-2v} = x \left(1 + \left(\frac{\lambda}{1000 v}\right)^2\right)^{-2v} \geq x/2.
\]
Hence
\[
\begin{split}
&\frac{1}{x}|\{(p_1,\dots,p_k)\in \mathcal{P}^k \colon x/2 \leq p_1 \dotsm p_k \leq x\}| \\
&\geq  \frac{1}{2}\sum_{\substack{(p_1,\dots,p_k)\in \mathcal{P}^k \\ x/2 \leq p_1 \dotsm p_k \leq x}} \frac{1}{p_1 \dotsm p_k} \geq \frac{1}{2}\sum_{\substack{  (a_1,\dots,a_k)\in A^k\\  N-k < a_1+\dots+a_k \leq N}} \prod_{i=1}^k \sum_{\substack{p \in \mathcal{P} \\ \rho^{a_i} \leq p < \rho^{a_i+1}}} \frac{1}{p} \\
&\geq  \frac{1}{2}\sum_{\substack{  (a_1,\dots,a_k)	\in A^k\\   N-k < a_1+\dots+a_k \leq N}} \frac{e^{-kj_0}}{a_1 \dotsm a_k} \geq \frac{1}{2}\frac{e^{-kj_0}}{(N/u)^k} \sum_{\substack{  (a_1,\dots,a_k)	\in A^k\\   N-k < a_1+\dots+a_k \leq N}} 1 \\
&\geq \frac{\alpha_v}{2N} \cdot  \left(\frac{e^{-j_0} |A|}{N/u}\right)^k \geq \frac{\alpha_v}{v^{2} \log x} \cdot \frac{1}{e^{O_\lambda(k)} (v \log v)^k} \geq \frac{\pi_v}{v^k \log x},
\end{split}
\]
where $\pi_v = \frac{\alpha_v}{v^{2} e^{O_\lambda(k)}(\log v)^{v}} = v^{-o_\lambda(v)}$ as $v \to \infty$.
\end{proof}

\subsection{Reduction of Hypothesis A to Hypothesis A*}
In this section we reduce Hypothesis A (except for the last claim concerning the case $u$ is fixed) into a variant where $u \asymp v$. Let $u_0 = 3/\lambda$. We claim that, under the assumptions of Hypothesis A, there is some $j$ such that
\[
\sum_{\substack{a \in A \\ N/(u_0^{j+1} u) < a \leq N/(u_0^{j} u)}} \frac{1}{a} \geq \frac{1+\lambda/3}{u_0^j u}.
\]
This follows since otherwise, summing over $j \geq 0$ we get that
\[
\sum_{\substack{a \in A}} \frac{1}{a} < \frac{1+\lambda/3}{u(1-1/u_0)} < \frac{1+\lambda}{u}
\]
which is a contradiction. Hence, Hypothesis A, except for the last claim concerning the case $u$ is fixed (which will be proved in Section~\ref{sec:ufixed}), follows if we prove the claim when $A \subseteq (\lambda N/ u, N/u]$, i.e. if we prove the following hypothesis.
\theoremstyle{plain}
\newtheorem*{hypa2}{Hypothesis A*}
\begin{hypa2}
Fix $\lambda \in (0, 1)$. There exists a constant $c = c(\lambda)$ such that the following holds. Let $N \geq u \geq 1$ be such that $N \geq (10u/\lambda)^2$. Let $A$ be a subset of the integers in $(\lambda \frac{N}{u},\frac{N}{u}]$ such that
\[
\sum_{\substack{ a\in A  }}  \frac 1a \geq \frac {1+\lambda}{u}.
\]
Then there exists an integer $k\in[u, u/\lambda]$ such that
\[
|\{(a_1,\dots,a_k)\in A^k \colon N-k < a_1+\dotsb+a_k \leq N\}| \ge \alpha_{k,u} \frac{|A|^{k}}{N},
\]
where $\alpha_{k,u}$ is a constant with $\alpha_{k,u} = (c/\log u)^k$ as $u \to \infty$. 
\end{hypa2}

\section{Proving Hypothesis A*: an outline}
\label{sec:A*outline}

Our main goal has become to prove Hypothesis A*, a quantitative variant of Proposition~\ref{prop:DiscBlei}, concerning the number of solutions to $a_1 + \cdots + a_k = t$ for some fixed $t$. In Section~\ref{sec:A*outline1} we state some removal-type results in this spirit. When the number of variables $k$ is bounded, this follows from an arithmetic regularity lemma of Green \cite{Gr05}. However, when $k$ grows, the situation becomes different and we will prove the substitute Theorem \ref{thm:pop-rem-largek} in Sections~\ref{sec:fill}--\ref{sec:proofofpopremlaagek}. Hypothesis A* will be deduced from these results in Sections~\ref{sec:A*outline2} and~\ref{sec:A*outline3}.

\subsection{An arithmetic regularity lemma for popular sums}\label{sec:A*outline1}

An important tool in graph theory is the triangle removal lemma, which can be proved using Szemer\'{e}di's regularity lemma. Green \cite{Gr05} developed an arithmetic version of the regularity lemma, and deduced as a consequence a removal lemma in the arithmetic setting.

\begin{theorem}[Arithmetic removal lemma] \label{thm:arith-rem}
Let $k \geq 3$ be a positive integer. Let $G$ be a finite abelian group with $|G| = N$, and let $A_1,\cdots,A_k \subseteq G$ be subsets. For any $\eta>0$ there exists a positive constant $\delta = \delta(k,\eta) > 0$ such that the following statement holds. If the number of solutions to $a_1+\cdots+a_k = 0$ with $a_i \in A_i$ for all $i$ is at most $\delta N^{k-1}$, then for each $i$ there exists a subset $A_i' \subseteq A_i$ with $|A_i \setminus A_i'| \leq \eta N$, such that there is no solution to $a_1' + \cdots + a_k' = 0$ with $a_i' \in A_i'$ for all $i$.
\end{theorem}

By inspecting the proof, one notes that the construction of $A_i'$ is translation-invariant, in the sense that if $A_i = B_i + t$ for some $t \in G$, then one can take $A_i' = B_i' + t$. Using this observation, the following extension of this arithmetic removal lemma quickly follows.

\begin{theorem}[Removal lemma for popular sums] \label{thm:pop-rem}
Let $k \geq 3$ be a positive integer. Let $G$ be a finite abelian group with $|G| = N$, and let $A_1,\cdots,A_k \subseteq G$ be subsets. For any $\eta>0$ there exists a positive constant $\delta = \delta(k,\eta) > 0$ such that the following statement holds. For each $i$ there exists a subset $A_i' \subseteq A_i$ with $|A_i \setminus A_i'| \leq \eta N$, such that for each $x \in A_1'+\cdots+A_k'$, the number of solutions to $a_1+\cdots+a_k = x$ with $a_i \in A_i$ for all $i$ is at least $\delta N^{k-1}$.
\end{theorem}

\begin{proof}
Define $A_1',\cdots,A_k'$ as in Green's proof of the arithmetic removal lemma. Let $x \in A_1'+\cdots+A_k'$, and suppose that there are at most $\delta N^{k-1}$ solutions to $a_1+\cdots+a_k = x$ with $a_i \in A_i$ for all $i$. Theorem~\ref{thm:arith-rem} applied with the sets $A_1,\cdots,A_{k-1},A_k-x$, along with the observation made above about translation invariance, implies that there is no solution to $a_1' + \cdots + a_k' = x$ with $a_i' \in A_i'$ for all $i$, which is absurd.
\end{proof}

In other words, Theorem \ref{thm:pop-rem} asserts that, given a positive density subset $A \subseteq G$, all $k$-fold sums can be made popular by removing a few elements from $A$, for any fixed $k \geq 3$. When $k=2$, Theorem~\ref{thm:arith-rem} is trivially true whereas Theorem \ref{thm:pop-rem} fails (see~\cite{Wolf10} for a construction of a counterexample using niveau sets).

It was later observed in \cite{KSV09,Sz10} that Theorem \ref{thm:arith-rem} can also be deduced directly from the graph removal lemma, bypassing the arithmetic regularity lemma. In this way Theorem \ref{thm:arith-rem} can also be generalized to deal with general linear equations using hypergraph removal lemmas; see~\cite{Sh10} and references therein. On the other hand, in order to deduce Theorem \ref{thm:pop-rem} it seems necessary to use the arithmetic regularity lemma due to the translation-invariance property required. Consequently, while 
\cite[Conjecture 9.4]{Gr05} is proved in~\cite{Sh10}, its extension in the spirit of Theorem \ref{thm:pop-rem} is still open.

For subsets $A$ in an arbitrary abelian group (not necessarily finite), the following analogue of Theorem 3.2 can be deduced via a Freiman isomorphism.

\begin{corollary} \label{cor:pop-rem-torsion-free}
Let $k \geq 3$ be a positive integer. Let $G$ be an arbitrary abelian group, and let $A_1,\cdots,A_k \subseteq G$ be finite subsets. Let $A = A_1\cup \cdots \cup A_k$ and assume that $|A+A| \leq K|A|$ for some $K \geq 1$.
For any $\eta>0$ there exists a positive constant $\delta = \delta(k,\eta,K) > 0$ such that the following statement holds. For each $i$ there exists a subset $A_i' \subseteq A_i$ with $|A_i \setminus A_i'| \leq \eta |A|$, such that for each $x \in A_1'+\cdots+A_k'$, the number of solutions to $a_1+\cdots+a_k = x$ with $a_i \in A_i$ for all $i$ is at least $\delta |A|^{k-1}$.
\end{corollary}

\begin{proof}
By Freiman's theorem, there is a Freiman isomorphism $\pi: A \rightarrow \widetilde{G}$ of order $k$ from $A$ to a finite abelian group $\widetilde{G}$, with image $\pi(A) = \widetilde{A}$, such that $|\widetilde{A}| = \alpha |\widetilde{G}|$ for some $\alpha = \alpha(k,K) > 0$. Let $\widetilde{A}_i = \pi(A_i)$ for $1 \leq i\leq k$. By Theorem \ref{thm:pop-rem} applied to $\widetilde{A}_1,\cdots,\widetilde{A}_k$ (with $\eta$ replaced by $\eta\alpha$), there are subsets $\widetilde{A}'_i \subseteq \widetilde{A}_i$ with $|\widetilde{A}_i \setminus \widetilde{A}'_i| \leq \eta\alpha |\widetilde{G}| = \eta |A|$, such that for each $\widetilde{x} \in \widetilde{A}'_1 + \cdots + \widetilde{A}'_k$, the number of solutions to $\widetilde{a}_1 + \cdots + \widetilde{a}_k = \widetilde{x}$ with $\widetilde{a}_i \in \widetilde{A}_i$ is at least $\delta |\widetilde{G}|^{k-1}$ for some $\delta = \delta(k,\eta,K) > 0$.

Now let $A_i' = \pi^{-1}(\widetilde{A}_i')$. Then $|A_i \setminus A_i'| = |\widetilde{A}_i \setminus \widetilde{A}_i'| \leq \eta |A|$. For any $x \in A_1'+\cdots+A_k'$, note that any solution to $\widetilde{a}_1 + \cdots + \widetilde{a}_k = \pi(x)$ with $\widetilde{a}_i \in \widetilde{A}_i$ gives rise to a solution to $a_1 + \cdots + a_k = x$ with $a_i = \pi^{-1}(\widetilde{a}_i)$, and moreover different solutions to the former give different solutions to the latter. The desired conclusion follows immediately.
\end{proof}

We expect some version of Corollary~\ref{cor:pop-rem-torsion-free} to hold as $k$ grows, and in this direction we will prove the following theorem.

\begin{theorem} \label{thm:pop-rem-largek}
For any $K \geq 1$ and $\eta > 0$, there exist positive integers $m = m(\eta, K)$ and $\ell = \ell(\eta, K)$ and a positive constant $\delta = \delta(\eta, K)$ such that the following statement holds.
Let $A \subseteq G$ be a subset in a torsion-free abelian group $G$ with $|A+A| \leq K|A|$. Then there exist an element $z \in G$ with $z + \ell A \subset (m+\ell)A$, and a subset $A' \subseteq A$ with $|A'| \geq (1-\eta)|A|$, such that for any positive integer $k > \ell$ and any element $x \in kA' + z$, we have $r_{(k+m)A}(x) \geq (\delta |A|)^{k+m-1}$, where $r_{nA}(x)$ denotes the number of solutions to $a_1 + \cdots + a_n = x$ with $a_1,\cdots,a_n \in A$. 

\end{theorem}

In the following two subsections we will show how the removal lemmas can be used to prove Hypothesis A*, and the proof of Theorem~\ref{thm:pop-rem-largek} will occupy Sections~\ref{sec:fill}--\ref{sec:proofofpopremlaagek}. To end this subsection, we give a rough sketch of the main ideas of the proof of Theorem \ref{thm:pop-rem-largek}, motivated by arguments in \cite{Ruz97}.

We shall first deduce a filling lemma: from the removal lemma for popular sums with a fixed number of summands and work of Tao and Vu~\cite{TV08} we deduce that 
there is a bounded $m$ and a proper progression $P$ such that $A \subseteq P$ and $mA$ (popularly) contains a translate of $P$, possibly after removing a small proportion of elements from $A$.

Now write $C$ for the convex hull of $A$, so that $C \subseteq P$. After shrinking $A$ a bit, any element $x \in kA$ is then a popular sum in $kC$. We then use an induction and the Shapley-Folkman theorem (see Lemma~\ref{le:ShapFolk} below) to show that popular sums in $kC$ are also popular in $(k-1)C + A$ (if $C$ is slightly shrunk in an appropriate way). After doing this reduction enough times, we deduce that $x$ is popular in $rC + (k-r)A$, for some bounded $r$. The final task of finding popular representations of elements in $rC$ can be done through the filling lemma described above since $rmA$ popularly contains a translate of $rC$. In practice we need to be very careful to always guarantee popularity at each stage.

\subsection{Proof of Hypothesis A* for bounded $u$}\label{sec:A*outline2}

We divide into two cases depending on whether $u = O_{\lambda}(1)$ or not. First suppose that $u = O_{\lambda}(1)$. Since Theorem~\ref{thm:pop-rem} is only applicable for $k \geq 3$, we need to do some initial preparations to handle the case where we would have $k = 2$. Write
\[
A' = \{a \in A \colon A \cap (N-2-a, N-a] = \emptyset\}.
\]
If $|A \setminus A'| \geq \left(\frac{\lambda}{2u}\right)^2 |A|$, the claim follows with $k = 2$, so we can assume that $|A \setminus A'| < \left(\frac{\lambda}{2u}\right)^2 |A|$. Then, by assumptions on set $A$, we have
$A' \subseteq (\lambda N/u,N/u]$ and
\[ 
\sum_{a \in A'} \frac{1}{a} \geq \sum_{a \in A} \frac{1}{a} - \frac{u}{\lambda N} |A\setminus A'| \geq \frac{1+3\lambda/4}{u},
\]
so that $A'$ has density at least $\lambda/u$ on the interval $[1,N/u]$. By Corollary \ref{cor:pop-rem-torsion-free} we may find a subset $A'' \subseteq A'$ with $|A' \setminus A''| \leq (\frac{\lambda}{2u})^2 |A'|$, such that for any $3 \leq k \leq u/\lambda$ and any $x \in kA''$, we have
\begin{equation}\label{eq:A*1} 
\left| \{(a_1,\cdots,a_k) \in A^k: a_1+\cdots+a_k = x\} \right| \geq \delta |A|^{k-1}, 
\end{equation}
for some $\delta = \delta(\lambda, u)>0$. Since
\[ 
\sum_{a \in A''} \frac{1}{a} \geq \sum_{a \in A'} \frac{1}{a} - \frac{u}{\lambda N} |A'\setminus A''| \geq \frac{1+\lambda/2}{u} > \frac{1}{u} + \frac{1}{\sqrt{N}-1} 
\]
by the lower bound for $N$, Bleichenbacher's theorem (Proposition \ref{prop:DiscBlei}) implies that there exists a positive integer $k$ and $a_1',\cdots,a_k' \in A''$ such that $N-k < a_1'+\cdots+a_k' \leq N$. Note that we necessarily have $k \in [u,u/\lambda]$, and by the choice of $A'$, we must have $k \neq 2$. If $k=1$, the claim follows immediately. If $k \geq 3$, then~\eqref{eq:A*1} applied to $a_1'+\cdots+a_k'$ gives that
\[ \left| \{(a_1,\cdots,a_k) \in A^k: N-k < a_1+\cdots+a_k \leq N \} \right| \geq \delta |A|^{k-1}, \]
as desired.

\subsection{Proof of Hypothesis A* for large $u$, assuming Theorem \ref{thm:pop-rem-largek}}\label{sec:A*outline3}
For the rest of the proof assume that $u\geq U$ for some sufficiently large $U$ depending on $\lambda$.  Let us now prove by induction on $j \geq 0$ that Hypothesis A* holds when $2^j U \leq u \leq 2^{j+1} U$. Let $A \subseteq (\lambda N/u, N/u]$ be a subset with $\sum_{a \in A} \frac{1}{a} > (1+\lambda)/u$. We wish to find a positive integer $k \in [u, u/\lambda]$ such that
\begin{equation}\label{eq:largek-induct}
\left|\{(a_1, \dots, a_k) \in A^k : N-k < a_1 + \dots + a_k \le N\}\right| \geq \left(\frac{c}{\log u} \right)^k \cdot \frac{|A|^k}{N}
\end{equation}
for some $c = c(\lambda) > 0$. In case $j = 0$, this follows from the work on case $u = O_\lambda(1)$ once $c$ is small enough.

Assume now that $2^j U \leq u \leq 2^{j+1} U$ for some $j \geq 1$. We shall study popular doubling in $A$, but first we need to find an appropriate notion of popularity. Write $r_0 = 0$ and
\[ 
r_i = 2^{i-10} \lambda^4 \frac{|A|^2}{|2A|} 
\]
for $i \geq 1$. For $i \geq 0$, let
\[ 
B_i = \{n \in 2A: r_i < r_{2A}(n) \leq r_{i+1}\}. 
\]

Note that since $\frac{|A|}{|2A|} \geq \frac{\lambda N /u^2}{2N/u} = \frac{\lambda}{2u}$ and $r_{2A}(n) \leq |A|$ for all $n \in 2A$, the set $B_i$ is empty for $i \geq I = 2 \log u - 1$. Furthermore
\[
\sum_{0 < i < I} |B_i| r_i = \frac{1}{2} \sum_{0 \leq i < I} |B_i| r_{i+1} - \frac{1}{2}|B_0| r_1 \geq \frac{1}{2} \sum_{n \in 2A} r_{2A}(n) - \frac{1}{2}|2A| \cdot \frac{\lambda^4 |A|^2}{2^9|2A|} \geq \frac{|A|^2}{4},
\]
so that there exists a smallest positive index $i_0$ such that
\[
|B_{i_0}| r_{i_0} \geq \frac{\lambda^4}{512 \log u} |A|^2.
\]
We choose
\[
E = \left\{(a_1, a_2) \in A \times A \colon r_{2A}(a_1+a_2) > r_{i_0} \right\}.
\]
Now
\begin{equation}\label{eq:Ecomplement}
|(A \times A) \setminus E| \leq \sum_{i < i_0} |B_i| r_{i+1} \leq |B_0|r_1 + 2\log u \cdot 2 \cdot \frac{\lambda^4}{512 \log u} |A|^2 \leq \frac{\lambda^4}{64}|A|^2.
\end{equation}
Write $D = \cup_{i \geq i_0} |B_i| = A {\overset{E}{+}} A \subseteq (2\lambda N/u, 2N/u]$.

Let us first consider the case that $|D| > 8  N/u^2$. Then 
\[
\sum_{\substack{d \in D \\ \lambda \frac{N}{u/2} < d \leq \frac{N}{u/2}}} \frac{1}{d} \geq |D| \frac{u}{2N} > \frac{1+\lambda}{u/2},
\]
and thus by induction hypothesis there is an integer $k/2 \in [u/2, u/(2\lambda)]$ such that
\[
|\{(d_1, \dots, d_{k/2}) \in D^{k/2} : N- k/2 < d_1 + \dotsb + d_{k/2} \le N\}| \geq \left(\frac{c}{\log u}\right)^{k/2} \frac{|D|^{k/2}}{N}.
\]
Hence, by the definitions of $D$ and $i_0$, we have
\[
\begin{split}
&|\{(a_1, \dots, a_{k}) \in A^{k} : N-k < a_1 + \dotsb + a_{k} \leq N\}| \\
&\geq |\{(d_1, \dots, d_{k/2}) \in D^{k/2} : N-k/2 < d_1 + \dotsb + d_{k/2} \leq N\}| \cdot r_{i_0}^{k/2}\\
&\geq  \left( \frac{c}{\log u} \right)^{k/2} \frac{|B_{i_0}|^{k/2}}{N} \cdot r_{i_0}^{k/2} 
\geq  \left( \frac{c}{\log u} \right)^{k/2} \cdot \frac{1}{N} \left(\frac{\lambda^4}{512 \log u} |A|^2 \right)^{k/2} \geq \left(\frac{c}{\log u}\right)^k \frac{|A|^k}{N},
\end{split}
\]
provided that $c \leq \frac{\lambda^4}{512}$.

Let us now consider the case that $|D| \leq 8 N/u^2$. We need the following lemma.

\begin{lemma}\label{le:BSG}
Let $(G, +)$ be an abelian group and let $\delta > 0$. If $E \subseteq A \times A$ satisfies
\[
|E| \ge (1-\delta^2)|A|^2 \quad \text{and} \quad |A \Epl A| \le K|A|,
\]
then there exists a set $A' \subseteq A$ such that
\[
|A'| \ge (1-2\delta)|A| \quad \text{and} \quad |A'+A'| \le \frac{K^3}{1-6\delta}|A|.
\]
\end{lemma}
\begin{proof}
This is a variant of the Balog-Szemer\'edi-Gowers theorem (see \cite[Theorem 2.29]{TV10}) which can be proved by incorporating the hint for \cite[Exercise 2.5.4]{TV10} to the proof of the Balog-Szemer\'edi-Gowers theorem in \cite[Section 6.4]{TV10}. See also \cite[Lemma 5.1]{GrKoMa} for a proof of a variant for $A-A$.
\end{proof}

Since $|A| \geq \lambda N/u^2$ and $|D| \leq 8N/u^2$, we have $|A \overset{E}{+} A| \leq (8/\lambda) |A|$. Recall also \eqref{eq:Ecomplement}. Thus Lemma \ref{le:BSG} implies that there is a subset $B \subseteq A$ such that
\[ 
|B| \geq (1-\lambda^2/4) |A| \quad \text{and} \quad |B+B| \le (20/\lambda)^3 |B|.
\]
Applying Theorem \ref{thm:pop-rem-largek} to $B$ with $K = (20/\lambda)^3$ and $\eta = \lambda^2/4$, we obtain an element $z \in \Z$ with $|z| = O_{\lambda}(N/u)$ and a subset $B' \subseteq B$ with $|B'| \geq (1-\eta) |B|$ satisfying the property that whenever $x \in kB' + z$ for any positive integer $k$ large enough depending on $\lambda$, there exists $n \in [k, k + O_{\lambda}(1)]$ such that $r_{nB}(x) \geq (c|B|)^{n-1}$. 

Since $B' \subseteq (\lambda N/u, N/u]$, we have
\[ 
\begin{split}
\sum_{b \in B'} \frac{1}{b} &\geq \sum_{a\in A}\frac{1}{a} - \frac{u}{\lambda N} \left( |A\setminus B| + |B\setminus B'| \right) \geq \sum_{a \in A} \frac{1}{a} - \frac{u}{\lambda N} \cdot \frac{\lambda^2}{2} |A| \\ 
& \geq \sum_{a \in A} \frac{1}{a} - \frac{u}{\lambda N} \cdot \frac{\lambda^2}{2} \cdot \frac{N}{u} \sum_{a \in A} \frac{1}{a} \geq \left(1-\frac{\lambda}{2}\right)\sum_{a \in A} \frac{1}{a} > \frac{1+\lambda/3}{u}. \\ 
\end{split}
\]
Recalling that $|z| =  O_{\lambda}(N/u)$, and writing $N' = N-z$ and $u' = u (N-z)/N$, we have
\[
u' = u - \frac{uz}{N} \geq u - O_{\lambda}(1) \geq \left(1-\frac{\lambda}{10}\right) u
\]
provided that $U$ is large enough. Hence
\[
\sum_{\substack{b \in B' \\ b \leq N'/u'}} \frac{1}{b} > \frac{1+\lambda/3}{u} \geq  \frac{1+\lambda/6}{u'} > \frac{1}{u'} + \frac{1}{\sqrt{N'}-1}.
\]

Hence, by discrete Bleichenbacher's theorem (Proposition \ref{prop:DiscBlei}), we find $k$ and $b_1,\cdots,b_k \in B'$ such that
\[ 
N-z-k < b_1 + \cdots + b_k \leq N-z. 
\]

Write $x := b_1+\cdots+b_k+z \in (N-k, N]$. Now $x \in kB'+z$ and hence there exists $\ell \in [k, k+O(1)]$ such that $r_{\ell B}(x) \geq (c|B|)^{\ell-1}$. Therefore,
\[ |\{(a_1, \dots, a_{\ell}) \in A^{\ell} : N-\ell < a_1 + \dotsb + a_{\ell} \leq N\}|   \geq r_{\ell A}(x) \geq r_{\ell B}(x) \geq (c|B|)^{\ell-1} \geq (c|A|/2)^{\ell-1}. \]
This clearly implies \eqref{eq:largek-induct}, completing the proof of Hypothesis A*.

\section{The filling argument}
\label{sec:fill}

In this section we carry out the first step in proving Theorem \ref{thm:pop-rem-largek}, that of locating a proper progression $P$ containing $A$ such that $mA$ fills a translate of $P$ for some bounded $m$.

\begin{lemma}[Filling lemma]\label{lem:fill-unpop}
For any $K \geq 1$, there exists a positive integer $m = m(K)$ such that the following statement holds. Let $G$ be a torsion-free abelian group, and let $A \subseteq G$ be a finite subset with $|A+A| \leq K|A|$. Then there is a proper progression $Q$ of rank $O_K(1)$ with size $|Q| = O_K(|A|)$, such that $A \subseteq Q$ and $g + Q \subseteq mA$ for some $g \in G$.
\end{lemma}

\begin{proof}
By Freiman's theorem there is a proper progression $P$ of rank $d-1 = O_K(1)$ containing $A$, such that $|A| = \alpha |P|$ for some $\alpha \gg_K 1$. Thus for any positive integer $\ell$, we have
\[ | \ell A| \leq | \ell P| \leq \ell^{d-1} |P|  \leq \alpha^{-1} \ell^{d-1} |A|. \]
The hypotheses in \cite[Theorem 1.21]{TV08} are then satisfied for $\ell$ large enough depending on $K$. Hence there is a proper progression $Q'$ of rank $d' \leq d-1$, such that
\[ g + Q' \subseteq \ell A \subseteq g' + kQ', \]
for some constant $k = k(d)$ and some $g,g' \in G$. Hence the iterated sumset $k\ell A$ contains a translate of $kQ'$, which in turn contains a translate of $A$. Finally, by \cite[Corollary 1.11]{TV08} we may find a proper progression $Q$ containing $kQ'$, such that $Q$ is contained in $jkQ'$ for some $j = j(d)$. Thus for $m' = jk\ell$, the iterated sumset $m' A$ contains a translate of $Q$, which in turn contains a translate of $A$. Since $d = O_K(1)$ and $\ell, j, k = O_{d, K}(1)$, we have $m' \leq m$ for some integer $m$ depending only on $K$. Clearly the claim holds for this $m$.
\end{proof}

Combining the previous lemma with Corollary~\ref{cor:pop-rem-torsion-free} we obtain the following filling lemma for popular sums.
\begin{lemma}[Filling lemma, popularity version]\label{lem:fill}
For any $K \geq 1$ and $\eta \in (0, 1/2)$, there exist a large positive integer $m = m(K)$ and a small positive constant $\delta = \delta(K, \eta)$ such that the following statement holds.
Let $A \subseteq G$ be a subset in a torsion-free abelian group $G$ with $|A+A| \leq K|A|$. Then there exist a proper progression $P$ of rank $O_K(1)$ with size $|P| = O_K(|A|)$ and a subset $A' \subseteq A$ with $|A'| \geq (1-\eta)|A|$ and $A' \subseteq P$, such that $mA'$ popularly contains some translate of $P$. That is, for some $g \in G$ we have 
 $r_{mA}(x) \geq \delta |A|^{m-1}$ for any $x \in g+P$.
\end{lemma}

\begin{proof}
Let $m = m(2K)$ be the constant from Lemma~\ref{lem:fill-unpop}. By Corollary~\ref{cor:pop-rem-torsion-free}, there is a subset $A' \subseteq A$ with $|A'| \geq (1-\eta) |A|$ such that, for each $x \in m A'$ we have $r_{mA}(x) \geq \delta |A|^{m-1}$ for some $\delta=\delta(K, \eta)>0$. Since $|A'+A'| \leq 2K|A'|$, Lemma \ref{lem:fill-unpop} implies that there is a proper progression $P$ of rank $O_K(1)$ with size $|P| = O_K(|A|)$, with the properties that $A' \subseteq P$ and $g + P \subseteq mA'$ for some $g \in G$. For each $x \in g+P$ we then have $x \in mA'$, and hence $r_{mA}(x) \geq \delta |A|^{m-1}$, as desired.
\end{proof}


\begin{remark}
It is a standard result in additive combinatorics that $3A$ contains a large progression $P$ of small rank. Here we require the extra condition that $A$ is (essentially) contained in a translate of $P$. A similar result is proved in \cite[Lemma 2.5]{SV06}, but it is not enough for us to deduce Theorem \ref{thm:pop-rem-largek}.

When $G = \Z$ and $A$ lies densely inside an interval, Lemma~\ref{lem:fill} can also be proved by a Fourier analytic argument (see \cite[Lemma 8.5]{Bo97}).
\end{remark}

Via the filling lemma (Lemma~\ref{lem:fill}), Theorem \ref{thm:pop-rem-largek} reduces to the following proposition.

\begin{proposition} \label{prop:pop-rem-largek}
For any $\alpha, \eta > 0$ and $d \in \mathbb{N}$, there exist a positive integer $\ell = \ell(\alpha,d,\eta)$ and a small positive constant $\delta = \delta(\alpha,d,\eta)$ such that the following statement holds. Let $P = ([-N_1,N_1]\times\cdots\times [-N_d,N_d]) \cap \Z^d$ be a box for some positive integers $N_1,\cdots,N_d$, and let $A \subseteq P$ be a subset with $|A| \geq \alpha |P|$. Then there exists a subset $A' \subseteq A$ with $|A'| \geq (1-\eta)|A|$ with the property that, for any positive integer $k > \ell$ and any element $x \in kA'$, there are at least $(\delta |A|)^{k-\ell}$ ways to write $x = y + a_1 + \cdots + a_{k-\ell}$ with $y \in \ell P$ and $a_1,\cdots,a_{k-\ell} \in A$.
\end{proposition}

\begin{proof}[Proof of Theorem \ref{thm:pop-rem-largek} assuming Proposition \ref{prop:pop-rem-largek}] 
First note that if Theorem \ref{thm:pop-rem-largek} holds for some subset $A$, then it also holds for any translate of $A$. From the filling lemma (Lemma~\ref{lem:fill}), we obtain a subset $A_1 \subseteq A$ with $|A_1| \geq (1-\eta/4)|A|$ and a proper progression $P$ of rank $d=O_K(1)$ with size $|P| = O_K(|A|)$, such that $A_1 \subset P$ and $nA_1$ popularly contains a translate of $P$ for some $n = n(\eta, K)$. By translating $A$ appropriately, we may further assume that $P$ is symmetric.

Let $\pi: \Z^d\rightarrow G \supset P$ be the Freiman homomorphism mapping the standard basis vectors in $\Z^d$ to the generators of $P$. Since $P$ is proper, the map gives a bijection between the box $\pi^{-1}(P)$ and $P$. Write $\widetilde{P} = \pi^{-1}(P)$ and $\widetilde{A}_1 = \pi^{-1}(A_1)$. Applying Proposition \ref{prop:pop-rem-largek} to the box $\widetilde{P}$ and the subset $\widetilde{A}_1$, we obtain a positive integer $\ell = \ell(\eta,K)$ and a subset $\widetilde{A}' \subseteq \widetilde{A}_1$ with $|\widetilde{A}'| \geq (1-\eta/2) |\widetilde{A}_1|$, such that for any $k > \ell$ and $\widetilde{x} \in k\widetilde{A}'$, the number of ways to write
\begin{equation}\label{eq:tildex} 
\widetilde{x} = \widetilde{y}+\widetilde{a}_1 + \cdots + \widetilde{a}_{k-\ell} \end{equation}
with $\widetilde{y} \in \ell \widetilde{P}$ and $\widetilde{a}_1,\cdots,\widetilde{a}_{k-\ell} \in \widetilde{A}_1$ is at least $(\delta |A|)^{k-\ell}$, for some positive constant $\delta = \delta(\eta,K) > 0$. 

Let $A' = \pi(\widetilde{A}')$. Clearly $|A_1\setminus A'| = |\widetilde{A}_1\setminus \widetilde{A}'| \leq (\eta/2) |A|$, and thus $|A\setminus A'| \leq \eta |A|$. Moreover, for any $k > \ell$ and $x \in kA'$, we may find  $\widetilde{x} \in k\widetilde{A}'$ such that $\pi(\widetilde{x}) = x$. Via the map $\pi$, each representation for $\widetilde{x}$ of the form \eqref{eq:tildex} gives rise to a representation $x$ of the form
\begin{equation}\label{eq:x} 
x = y + a_1 + \cdots + a_{k-\ell} 
\end{equation}
with $y \in \ell P$ and $a_1,\cdots,a_{k-\ell} \in A_1$. Hence there are at least $(\delta |A|)^{k-\ell}$ such representations for $x$.

Recall from the output of the filling lemma that $n A_1$ popularly contains a translate of $P$. It then easily follows (for example from~\cite[Lemma 5.3]{GrKoMa}) that $2\ell n A_1$ popularly contains a translate $z+\ell P$ for some $z \in G$. Thus each representation for $x \in kA'$ of the form \eqref{eq:x} gives rise to at least $(\delta |A|)^{2\ell n-1}$ ways to write $z+x$ as a sum of $2\ell n + (k-\ell)$ elements of $A$, since
\[ r_{2\ell n A}(z+y) \geq (\delta |A|)^{2\ell n-1} \]
if $\delta > 0$ is small enough. We conclude that for any $k > \ell$ and $x \in kA'$, we have
\[ r_{(2\ell n + k-\ell)A}(z+x) \geq (\delta |A|)^{k-\ell} (\delta |A|)^{2\ell n-1} = (\delta |A|)^{2\ell n + k - \ell - 1}. \]
This shows that Theorem~\ref{thm:pop-rem-largek} holds with this choice of $\ell$ and with $m = (2n-1)\ell$.
\end{proof}

We will prove Proposition~\ref{prop:pop-rem-largek} in Section~\ref{sec:proofofpopremlaagek} using geometrical ideas, after establishing some preliminary lemmas in Section~\ref{sec:convexbodies}.

\section{Auxiliary results about convex bodies}
\label{sec:convexbodies}

{\em Notations.} In this section and the next, we use normal letters such as $A, C, P$ to denote subsets of $\Z^d$, and boldface letters such as $\CC, \CP$ to denote convex bodies in $\R^d$. For $t>0$, we use $t\CC, t\CP$ to denote dilations of convex bodies in the usual manner.

\smallskip
The aim of this section is to prove two intuitive properties of convex hulls of positive density subsets $A$ of large boxes $P$ in $\Z^d$. The first one, Lemma~\ref{lem:contain-box} says that, for some constant $\varepsilon > 0$ (depending only on the density of $A$ and on the dimension), the convex hull of $A$ contains a translate of a small dilate of the convex hull of the box. The second one, Lemma~\ref{lem:small-boundary} states that most lattice points in the convex hull of $A$ are away from the boundary of the convex hull. 

Before stating and proving these, we state two auxiliary results which will be used in the proofs of the two lemmas.

%
%

\begin{lemma}
\label{le:vol>>|latticepoints|}
For any $\alpha > 0$ and positive integer $d$, there exist constants $N_0 = N_0(\alpha, d)$ and $c = c(\alpha, d)$ such that the following statement holds. Let $\CP = [-N_1,N_1] \times \cdots \times [-N_d,N_d]$ be a box in $\R^d$ for some positive integers $N_1,\cdots,N_d \geq N_0$, and let $P = \CP \cap \Z^d$. Let $A\subseteq P$ be a subset with $|A| \geq \alpha |P|$, and let $\CC \subseteq \R^d$ be the convex hull of $A$. Then $\vol(\CC) \geq c |P|$.
\end{lemma}

\begin{proof}
We can clearly assume that $\alpha < 1/100$. Write $M_j = \lceil \alpha^{3}N_j\rceil$, and let us split the box $P$ into $M_1 \dotsm M_d$ fibers
\[
P_{i_1, \dotsc, i_d} = \{(x_1, \dotsc, x_d) \in P \colon x_j \equiv i_j \pmod{M_j}\} \quad \text{with $0 \leq i_j < M_j$ for each $1 \leq j \leq d$},
\]
and write $A_{i_1, \dotsc, i_d} = A \cap P_{i_1, \dotsc i_d}$. For some $i_1, \dotsc, i_d$ we must have $|A_{i_1, \dotsc, i_d}| \geq \alpha |P_{i_1, \dotsc i_d}|$. By the natural bijection 
\[
\begin{split}
\rho \colon P_{i_1, \dotsc, i_d} &\to \prod_{j=1}^d \left[\left\lceil \frac{-N_j - i_j}{M_j}\right\rceil, \left\lfloor \frac{N_j - i_j}{M_j}\right\rfloor\right] \cap \mathbb{Z}^d =: P' \\ 
(x_1, \dotsc, x_d) &\to \left(\frac{x_1-i_1}{M_1}, \dotsc, \frac{x_d-i_d}{M_d}\right)
\end{split}
\]
we can map the corresponding fibre into the box $P'$ which has bounded sidelengths. Write $B$ for the image of $A_{i_1, \dotsc, i_d}$, and write $L_1,\cdots,L_d$ for the side lengths of $P'$. By our choice of $M_j$, we have $L_j > \alpha^{-1}$ for each $j$, once $N_0$ is large enough in terms of $\alpha$. It follows that
\[
|B| \geq \alpha |P'| > \max_j \prod_{i = 1, i \neq j}^d L_i.
\] 
Thus $B$ cannot be contained in any $d-1$-dimensional hyperplane, and so $B$ contains $d+1$ points generating a non-trivial simplex $\Delta$,  whose volume is at least $1/d!$ since its vertices are lattice points. Hence $\rho^{-1}(\Delta)$ has volume at least
\[
\frac{1}{d!} \prod_{i=1}^d M_i \gg_{\alpha,d} |P|,
\]
and the claim follows since, by convexity, $\rho^{-1}(\Delta) \subseteq \CC$.
\end{proof}




\begin{theorem}[John]
Let $\CC \subseteq \R^d$ be a convex body. There exists an invertible linear transformation $T: \R^d\rightarrow \R^d$ and a point $x_0 \in \CC$ such that
\[ B_d \subseteq T(\CC-x_0) \subseteq d B_d, \]
where $B_d$ is the unit ball $\{(x_1,\cdots,x_d):x_1^2+\cdots+x_d^2 \leq 1\}$.
\end{theorem}

\begin{proof}
See \cite[Theorem 3.13]{TV10}.
\end{proof}

\begin{lemma}[Large boxes inside convex sets]\label{lem:contain-box}
For any $\alpha>0$ and positive integer $d$, there exist $N_0 = N_0(\alpha,d) > 0$ and $\beta = \beta(\alpha,d)>0$ such that the following statement holds.
Let $\CP = [-N_1,N_1] \times \cdots \times [-N_d,N_d]$ be a box in $\R^d$ for some positive integers $N_1,\cdots,N_d \geq N_0$, and let $P = \CP \cap \Z^d$. Let $A\subseteq P$ be a subset with $|A| \geq \alpha |P|$. Let $\CC \subseteq \R^d$ be the convex hull of $A$. Then $x_0 + \beta\CP \subseteq \CC$ for some $x_0 \in \Z^d$.
\end{lemma}

\begin{proof}
Since $N_1,\cdots,N_d$ are large enough, Lemma~\ref{le:vol>>|latticepoints|} implies that
 \[ \vol(\CC) \geq c N_1N_2\cdots N_d \]
for some positive constant $c = c(\alpha, d)$.
 Now apply John's theorem to $\CC$ to obtain an invertible linear transformation $T: \R^d \rightarrow \R^d$ and a point $x_0 \in \CC$ such that
 \[ B_d \subseteq T(\CC-x_0) \subseteq dB_d. \]
 In particular, we have 
\[ T^{-1}(B_d) \subseteq \CC - x_0 \subseteq 2\CP, \]
and thus $T^{-1}(e_i) \in 2\CP$ for the standard basis vectors $e_1,\cdots,e_d$, so that the $(i,j)$-entry of $T^{-1}$ is $O(N_i)$ for each $1\leq i,j \leq d$. Moreover, we have 
 \[ (\det T)(\vol (\CC)) = \vol (T(\CC-x_0)) \leq \vol(dB_d), \]
so that $\det T \ll |P|^{-1}$ and $\det T^{-1} \gg |P|$.

Now consider the $(i,j)$-entry of $T$. The bounds on the matrix entries of $T^{-1}$ imply that the determinant of the $(j,i)$-minor of $T^{-1}$ is $O(|P|/N_j)$. It follows that the $(i,j)$-entry of $T$ is bounded in absolute by
\[ O \left( \frac{1}{\det T^{-1}} \cdot \frac{|P|}{N_j} \right) = O\left( \frac{1}{N_j} \right). \]
It follows that $\|T(x)\|_{\infty} \ll 1$ for any $x \in \CP$. Hence $T(\CP) \subseteq \beta^{-1}B_d$ for $\beta>0$ small enough. This implies that $\beta\CP \subseteq T^{-1}(B_d)$ and thus $x_0 + \beta\CP \subseteq \CC$ as desired.

Finally, to ensure that $x_0 \in \Z^d$, we may replace $\beta$ by $\beta/2$ and note that $x_0 + (\beta/2)\CP$ contains a lattice point for any $x_0 \in \R^d$, once $N_0$ is large enough.
\end{proof}

\begin{lemma}[Lattice points near the boundary]\label{lem:small-boundary}
For any $\beta,\eta \in (0,1)$ and positive integer $d$, there exist $N_0 = N_0(d,\beta,\eta)>0$ and $\gamma = \gamma(d,\eta)>0$ such that the following statement holds.
Let $\CP = [-N_1,N_1] \times \cdots \times [-N_d,N_d]$ be a box in $\R^d$ for some positive integers $N_1,\cdots,N_d \geq N_0$.  Let $\CC \subseteq \CP$ be a convex body, and assume that $\beta\CP \subseteq \CC$. Let $\CC' = (1-\gamma)\CC$. Then $|(\CC \setminus \CC') \cap \Z^d| \leq \eta\cdot \vol(\CC)$.
\end{lemma}

\begin{proof}
Choose $\gamma = \gamma(d, \eta)$ so small that $(1+\gamma)^d-(1-2\gamma)^d < \eta$, and let $X = (\CC \setminus \CC') \cap \Z^d$. Let $\CB \subseteq \R^d$ be the unit box $[-1/2,1/2]^d$. Note that
\[ \bigcup_{x \in X} (x+\CB) \subseteq (\CC\setminus\CC') + \CB . \]
Since the union above is a disjoint union, we have
\[ |X| \leq \vol((\CC \setminus \CC') + \CB). \]
The volume above is at most
\[ \vol(\CC + \CB) - \vol(\{x \in \CC': x+\CB \subseteq \CC'\}). \]
If $N_0$ is large enough depending on $\beta$ and $\gamma$, then $\CB \subseteq \gamma\beta\CP$ and thus $\CB \subseteq \gamma\CC$. It follows that
\[ |X| \leq \vol((1+\gamma)\CC) - \vol((1-2\gamma)\CC) = \left[ (1+\gamma)^d-(1-2\gamma)^d \right] \vol(\CC) \leq \eta \cdot \vol(\CC), \]
by our choice of $\gamma$.
\end{proof}

\section{Proof of Proposition~\ref{prop:pop-rem-largek}}
\label{sec:proofofpopremlaagek}

Recall the notations from the beginning of Section~\ref{sec:convexbodies}. The following result of Shapley and Folkman resembles a simpler and non-popular version of what we wish to prove.
\begin{lemma}[Shapley-Folkman]
\label{le:ShapFolk}
Let $d$ be a positive integer, let $B \subseteq \R^d$, and let $\CC$ be the convex hull of $B$. For any integer $k > d$, one has
\[
k\CC = d\CC + (k-d)B.
\]
\end{lemma}
\begin{proof}
See e.g.~\cite[Appendix 1]{EkTe76} or~\cite[Corollary on page 435]{Cassels 75}.
\end{proof}

In order to extend the previous lemma to popular representations, we need to introduce some notation.

\begin{definition}[$\varepsilon$-regular subsets]
Let $\CP = [-N_1,N_1] \times \cdots \times [-N_d,N_d]$ be a box in $\R^d$ for some positive integers $N_1,\cdots,N_d$, and let $P = \CP \cap \Z^d$ and $P_{\varepsilon} = \varepsilon\CP \cap \Z^d$ for $\varepsilon>0$. A subset $A \subseteq P$ is called $\varepsilon$-regular, if for each $a \in A$ the small box $a + P_{\varepsilon}$ centered around $a$ contains at least $\varepsilon |P_{\varepsilon}|$ elements in $A$.
\end{definition}

\begin{lemma}[Regularization]\label{lem:regularize}
Let $\CP = [-N_1,N_1] \times \cdots \times [-N_d,N_d]$ be a box in $\R^d$ for some positive integers $N_1,\cdots,N_d$, and let $P = \CP \cap \Z^d$. Let $A\subseteq P$ be a subset with $|A| = \alpha |P|$ for some $\alpha>0$. Let $\eta>0$ be real. For any $\varepsilon \in (0,\alpha \eta 100^{-d})$, there is an $\varepsilon$-regular subset $A' \subseteq A$ with $|A'| \geq (1-\eta) |A|$.
\end{lemma}

\begin{proof}
Let $P_{\varepsilon/2} = (\varepsilon/2)\CP \cap \Z^d$, so that the side lengths of $P_{\varepsilon/2}$ are precisely $2\lfloor \varepsilon N_1/2 \rfloor+1, \cdots, 2\lfloor \varepsilon N_d/2 \rfloor +1$. Cover $P$ by at most
\[ \prod_{i=1}^d \left\lceil \frac{2N_i+1}{2\lfloor \varepsilon N_i/2 \rfloor+1} \right\rceil \leq  10^d \prod_{i=1}^d \min(N_i,\varepsilon^{-1}) \]
translates of $P_{\varepsilon/2}$. Define $A'$ by removing from $A$ those translates containing at most $\varepsilon |P_{\varepsilon}|$ elements of $A$. Then $A'$ is $\varepsilon$-regular by construction, since any $a \in A'$ lies in a translate of $P_{\varepsilon/2}$ containing at least $\varepsilon |P_{\varepsilon}|$ elements of $A'$, but this translate of $P_{\varepsilon/2}$ is contained in $a + P_{\varepsilon}$.
 
Moreover, since
\[ \varepsilon |P_{\varepsilon}| \leq 10^d \varepsilon \prod_{i=1}^d \max(\varepsilon N_i,1), \]
the number of elements in $A\setminus A'$ is at most
\[ 100^d\varepsilon \prod_{i=1}^d \min(N_i,\varepsilon^{-1}) \max(\varepsilon N_i,1) \leq 100^d\varepsilon |P| \leq \eta |A|,  \]
as desired.
\end{proof}

\begin{proposition}[popular Shapley-Folkman]\label{prop:popular-sf}
Let $\CP = [-N_1,N_1] \times \cdots \times [-N_d,N_d]$ be a box in $\R^d$ for some positive integers $N_1,\cdots,N_d$, and let $P = \CP \cap \Z^d$. Let $A\subseteq P$ be a subset. Let $\CC \subseteq \R^d$ be the convex hull of $A$, and assume that $x_0 + \beta\CP \subseteq \CC$ for some $x_0 \in \Z^d$ and $\beta>0$. Let $\CC' = (1-\gamma)\CC + \gamma x_0$ for some $\gamma \in (0,1/(d+2))$.
 If $A$ is $\varepsilon$-regular for some $\varepsilon \leq \beta\gamma$, then for each positive integer $k > d$ and any element $x \in (k+1)\CC'$, there are at least $\delta |P|$ ways to write $x=y+a$ for some $y \in k\CC'$ and $a \in A$, where $\delta = \delta(d,\varepsilon) > 0$ is a positive constant.
\end{proposition}

\begin{proof}
By translation we may assume that $x_0 = 0$, so that $\beta\CP \subseteq \CC$ and $\CC' = (1-\gamma)\CC$. Let $x \in (k+1)\CC'$ be for some $k > d$. Write $x = (k+1)(1-\gamma)z$ for some $z \in \CC$. Note that
\[ x = (k-d-\gamma(k+1))z + (d+1)z = (k-d-\gamma(k+1))z + dw+ a \]
for some $a \in A$ and $w \in \CC$, by the Shapley-Folkman theorem (Lemma~\ref{le:ShapFolk}). We thus get a solution $x = y+a$ with
\[ y = (k-d-\gamma(k+1))z + dw. \]
We claim that 
\begin{equation}
\label{eq:yPepsinclclaim}
y+ P_{\varepsilon} \subset k\CC'.
\end{equation}
If this claim is true, then any $t \in P_{\varepsilon}$ with $a-t \in A$ gives rise to a representation
\[ x = (y + t) + (a - t) \]
with $y + t \in k\CC'$. By the $\varepsilon$-regularity of $A$, there are at least $\varepsilon |P_{\varepsilon}|$ such elements $t$, leading to at least
\[ \varepsilon |P_{\varepsilon}| \geq \varepsilon^{d+1} N_1\cdots N_d \geq \delta |P| \]
representations for some constant $\delta > 0$ depending on $d$ and $\varepsilon$, as desired.

To prove~\eqref{eq:yPepsinclclaim}, take any $t \in \varepsilon\CP$. Since $\beta\CP \subseteq \CC$ and $\varepsilon \leq \beta\gamma$, we have $\varepsilon\CP \subseteq \gamma\CC$, and thus $t \in \gamma\CC$. Hence
\[ y+t \in (k-d-\gamma(k+1))\CC + d\CC + \gamma\CC \subseteq k(1-\gamma)\CC = k\CC', \]
as desired.
\end{proof}


Now we are finally ready to prove Proposition~\ref{prop:pop-rem-largek}.
\begin{proof}[Proof of Proposition \ref{prop:pop-rem-largek}]
Let $\beta = \min_{d' \leq d} \beta(\alpha\eta/8,d')>0$, where $\beta(\cdot, \cdot)$ is from Lemma \ref{lem:contain-box}, and let $\gamma = \min\{1/(d+2), \min_{d' \leq d} \gamma(d', \eta/4)\} > 0$, where $\gamma(\cdot, \cdot)$ is from Lemma~\ref{lem:small-boundary}.  Let $N_0$ be the maximum of the constants $\max_{d' \leq d} N_0(\alpha\eta/8, d')$ from Lemma~\ref{lem:contain-box} and $\max_{d' \leq d} N_0( d', \beta, \eta/4)$ from Lemma~\ref{lem:small-boundary}.

Without loss of generality we may assume that $N_1,\cdots,N_{d'} \geq N_0$ and $N_{d'+1},\cdots,N_d < N_0$, for some $0 \leq d' \leq d$. We may also assume that $d'>0$ since otherwise $|A|$ is bounded and the conclusion is trivial. Let $\CP' = [-N_1,N_1]\times \cdot \times [-N_{d'},N_{d'}]$, all of whose side lengths are at least $2N_0 + 1$, and let $P' = \CP' \cap \Z^{d'}$. We can partition $P$ into $J \leq (2N_0+1)^{d-d'}$ smaller boxes $P_1,\cdots,P_J$, with $P_j = P' \times \{t_j\}$ for some $t_j \in \Z^{d-d'}$. For each $1 \leq j \leq J$, let $A_j \subset P'$ be the set of $a \in P'$ with $(a,t_j) \in A$. Let $\CJ$ be the set of indices $j$ with $|A_j| \geq (\eta\alpha/4) |P'|$.

Let $\varepsilon \in (0, \beta\gamma)$ be small enough depending on $\alpha,d,\eta$. For each $j \in \CJ$, Lemma \ref{lem:regularize} applied to $A_j \subseteq P'$ implies that there is an $\varepsilon$-regular subset $B_j \subseteq A_j$ with $|A_j\setminus B_j| \leq (\eta/4) |A_j|$. Let $\CC_j \subseteq \R^{d'}$ be the convex hull of $B_j$. Lemma \ref{lem:contain-box} applied to $B_j$ and $\CC_j$ implies that $x_j + \beta\CP' \subset \CC_j$ for some $x_j \in \Z^{d'}$.

For each $j \in \CJ$, we may thus apply Proposition~\ref{prop:popular-sf} to  $B_j$ and $\CC_j$ to define $\CC_j'$ and conclude that for any $k>d$ and $x \in (k+1)\CC_j'$, there are at least $\delta' |P'|$ ways to write $x = y+b$ for some $y \in k\CC_j'$ and $b \in B_j$, where $\delta' = \delta'(d,\varepsilon) > 0$ is a constant. This number of representations is at least $\delta |A|$ for $\delta = \delta'/J$.

Now for $j\in \CJ$ we write $A_j' = (B_j \cap \CC_j') \times \{t_j\}$ and let $A' = \cup_{j \in \CJ} A_j'$. To finish the proof, we show that the conclusion of Proposition~\ref{prop:pop-rem-largek} holds with this choice of $A'$ and $\ell = d(2N_0+1)^d \geq dJ$. Indeed, let $k > \ell$ and $x \in kA'$ be arbitrary. Assume that
\[ x = \sum_{j \in \CJ} (x_j, k_jt_j), \]
where $x_j \in k_j(B_j \cap \CC_j')$ and $\sum k_j = k$. We may choose $\ell_j \in [\min\{d,k_j\}, k_j]$ for $j \in \CJ$ such that $\sum \ell_j = \ell$. For those $j \in \CJ$ with $k_j > d$, by iterating the output of Proposition~\ref{prop:popular-sf} we see that the number of ways to write
 $ x_j = y_j + b_{j,1} + \cdots + b_{j,k_j-\ell_j}$ with $y_j \in \ell_j \CC_j'$ and $b_{j,1},\cdots,b_{j,k_j-\ell_j} \in B_j$ is at least $(\delta |A|)^{k_j - \ell_j}$. For those $j$ with $k_j \leq d$, we necessarily have $\ell_j = k_j$ and the statement above holds also. Hence we obtained at least $(\delta |A|)^{k-\ell}$ representations
\[ x = \sum_{j \in \CJ} (y_j, \ell_jt_j) + \sum_{j \in \CJ} \sum_{1 \leq i \leq k_j - \ell_j} (b_{j,i}, t_j) \]
of the desired form, since $\sum_{j \in \CJ} (y_j, \ell_jt_j) \in \ell P$ and each $(b_{j,i},t_j) \in A$.

To show that $A\setminus A'$ is small, observe from our constructions that
\[ |A\setminus A'| \leq \sum_{j\notin\CJ} |A_j| + \sum_{j\in\CJ} |A_j\setminus B_j| + \sum_{j\in\CJ} |(\CC_j \setminus \CC_j') \cap \mathbb{Z}^d|. \]
By the definition of $\CJ$, the first sum above is bounded by $(\eta\alpha/4)|P| \leq (\eta/4)|A|$. By the construction of $B_j$ from regularization and Lemma \ref{lem:small-boundary}, both the second and the third sums above are bounded by $\sum_{j\in\CJ} (\eta/4) |A_j| \leq (\eta/4)|A|$. This shows that $|A\setminus A'| \leq \eta |A|$, completing the proof.
\end{proof}

As shown in Section~\ref{sec:fill}, Proposition~\ref{prop:pop-rem-largek} implies Theorem~\ref{thm:pop-rem-largek}. Hence, as shown in Section~\ref{sec:A*outline}, this finishes the proof of Hypothesis A*. As shown in Section~\ref{sec:initred} this implies Hypothesis A and Theorem~\ref{th:MT} except for the last claims concerning the case $u$ is fixed.

\section{Case $u$ is fixed and $v \geq 1000u^2/\lambda^2$ of Hypothesis A}
\label{sec:ufixed}
In this section we deduce the last claim of Hypothesis A from the first part of Hypothesis A and the arithmetic removal lemma. Note that by Hypothesis A*, we know that the first part of Hypothesis A actually holds for $N \geq (30v/\lambda)^2$.

We can assume that $N$ is large enough depending on $u$ and $\lambda$, since otherwise the claim follows trivially from the  discrete Bleichenbacher theorem (Proposition~\ref{prop:DiscBlei}). Notice first that we can assume that, for every $u' \in [1, e^{-1/u}v]$, one has
\begin{equation}
\label{eq:u=1basicass}
\sum_{\substack{a \in A \\ N/e^{-1/u}v < a \leq N/u'}} \frac{1}{a} < \frac{1+\lambda/2}{u'},
\end{equation}
since otherwise the claim follows immediately from the first part of Hypothesis A. Notice also that
\begin{equation}
\label{eq:largeelA}
\sum_{\substack{(1-\frac{\lambda}{8})N < a \leq N/u}} \frac{1}{a} \leq \frac{\lambda}{4u}.
\end{equation}
Indeed, this is trivially true if $u \geq (1-\lambda/8)^{-1}$, and if $u \leq (1-\lambda/8)^{-1} \leq 8/7$, then each summand is at most $8/(7N)$ and there are at most $\lambda N/8 + 1$ summands. Using these we obtain that
\[
\begin{split}
\sum_{\substack{a \in A \\ \frac{\lambda}{8u} N < a \leq (1-\frac{\lambda}{8})N}} \frac{1}{a} &\geq \sum_{\substack{a \in A \\ N/v < a \leq N/u}} \frac{1}{a} - \sum_{\substack{N/v < a \leq N/e^{-1/u}v}} \frac{1}{a} - \sum_{\substack{a \in A \\ N/e^{-1/u}v < a \leq \frac{\lambda}{8u} N}} \frac{1}{a} - \sum_{\substack{(1-\frac{\lambda}{8})N < a \leq N/u}} \frac{1}{a} \\
&\geq \frac{1+\lambda}{u} - \frac{1+\lambda/8}{u} - \frac{\lambda}{4u} - \frac{\lambda}{4u} \geq  \frac{\lambda}{4u},
\end{split}
\]
where we used \eqref{eq:u=1basicass} to bound the third sum. This implies that 
\begin{equation}\label{eq:largeA'}
|A'| := \left|A \cap \left[\frac{\lambda}{8u} N < a \leq \left(1-\frac{\lambda}{8}\right)N\right]\right| \geq \frac{\lambda}{8 u} N \cdot \frac{\lambda}{4u} \gg N
\end{equation}
since $\lambda$ and $u$ are fixed. Then, by the removal lemma for popular sums (Theorem~\ref{thm:pop-rem}), there exist $\delta = \delta(\lambda)$ and $A'' \subseteq A'$ such that $|A''| \geq (1-\lambda^2/(1000u^2))|A'|$ and, for all positive integers $k \leq 8u/\lambda$, $r_{kA'}(n) \geq \delta N^{k-1}$ for every $n  \in kA''$.

Assume first that for some $k_0 \in \{1, \dotsc, \lfloor 8u/\lambda \rfloor\}$,
\begin{equation}
\label{eq:B'bound}
|B| := |k_0 A'' \cap [0.65 N, (1-\lambda/40)N]| \geq \frac{\lambda^3}{10000 u^3} N.
\end{equation}
Let $b \in B$. Writing $N' = N-b$, $v' = v(N-b)/N$ and $u' = e^{-1/u} v(N-b)/N$, we have, by~\eqref{eq:u=1basicass},
\[
\begin{split}
\sum_{\substack{a \in A \\ N'/v' < a \leq N'/u'}} \frac{1}{a}  &= \sum_{\substack{a \in A \\ N/v < a \leq N/e^{-1/u}v}} \frac{1}{a} \geq  \sum_{\substack{a \in A \\ N/v < a \leq N/u}} \frac{1}{a} -  \sum_{\substack{a \in A \\ N/e^{-1/u}v < a \leq N/u}} \frac{1}{a} \\
&\geq \frac{1+\lambda}{u} - \frac{1+\lambda/2}{u} = \frac{\lambda}{2u} \geq \frac{2}{u'}
\end{split}
\]
since $v \geq 1000u^2/\lambda^2$.

Note also that $N' = Nv'/v \geq (100v/\lambda)^2 v'/v \geq 30v'^2/\lambda^2$ since $v \geq v'$. Hence we can apply the first part of Hypothesis A with $N', u', v'$ obtaining that there exists $k_1 \leq v' \leq 0.35 v$ such that
\[
|\{(a_1,\dots,a_{k_1})\in A^{k_1} \colon N-b-k_1 \leq a_1+\dotsb+a_{k_1} \leq N-b\}| \ge \alpha_{v'}' \frac{|A|^{{k_1}}}{N},
\]
where $\alpha_{v'}' > 0$ is a constant with $\alpha_{v}' = v^{-o(v)}$ as $v \to \infty$. Write $k = k_1 + k_0 \leq 0.35v + 8u/\lambda \leq e^{-1/u} v$. We get that
\[
\begin{split}
&|\{(a_1,\dots,a_k)\in A^k \colon N-k \leq a_1+\dotsb+a_k \leq N\}|\\
& \ge \sum_{b \in B} r_{k_0 A'}(b) \cdot |\{(a_1,\dots,a_{k_1})\in A^{k_1} \colon N-b-k_1 \leq a_1+\dotsb+a_{k_1} \leq N-b\}| \\
&\geq \frac{\lambda^3}{10000u^3}N \cdot \delta N^{k_0-1} \cdot \alpha'_{v'} \frac{|A|^{k_1}}{N} \geq \alpha_v \frac{|A|^{k}}{N},
\end{split}
\]
where $\alpha_v = \frac{\lambda^3}{10000u^3} \cdot \delta \alpha'_{v'} = v^{-o_\lambda(v)}$ as $v \to \infty$.

Consider now the case that \eqref{eq:B'bound} does not hold for any $k_0 \leq 8u/\lambda$. Write 
\[
D = A'' \cap \left[\left(\frac{1}{2}-\frac{\lambda}{80}\right)N, 0.65 N\right].
\]
Then $|A'' \setminus D| \leq \frac{\lambda^2}{1250u^2} N$ since every $a \in A'' \setminus D$ produces some $k_0a$ counted in~\eqref{eq:B'bound} for some $k_0 \leq 8u/\lambda$. Since $|A''| \geq |A'|/2 \geq \frac{\lambda^2}{64u^2} N$ by~\eqref{eq:largeA'}, we have
\begin{equation}\label{eq:largeD} 
|D| \geq \frac{\lambda^2}{64u^2} N - \frac{\lambda^2}{1250u^2} N \geq \frac{\lambda^2}{100u^2} N. 
\end{equation}
In particular $D$ is non-empty. Since $A'' \subseteq (N/v, N/u]$, we must have $u \leq 2.1$. From the bound
\[ |A' \setminus D| \leq |A'' \setminus D| + |A' \setminus A''| \leq \frac{\lambda^2 N}{500u^2}, \]
together with~\eqref{eq:largeelA} we get
\[
\begin{split}
\sum_{\substack{a \in A \\  \frac{\lambda}{8u} N < a \leq N/u}} \frac{1}{a} &\leq \frac{8u}{\lambda N} |A' \setminus D| + \sum_{\substack{(1-\frac{\lambda}{8}) N < a \leq N/u}} \frac{1}{a} + \sum_{(\frac{1}{2}-\frac{\lambda}{80})N \leq a \leq 0.65 N} \frac{1}{a} \\
&\leq \frac{\lambda}{50u} + \frac{\lambda}{4u} + \log \frac{0.65}{1/2-\lambda/80} + O(1/N) < 0.3 + 0.27 \cdot \frac{\lambda}{u}.
\end{split}
\]

Let $d \in D$ and write now $N' = N-d, u' = \frac{8u}{\lambda} \cdot \frac{N-d}{N}$ and $v' = 1.9 v e^{-1/u} \frac{N-d}{N}$. Then
\[
\begin{split}
\sum_{\substack{a \in A \\ N'/v' < a \leq N'/u'}} \frac{1}{a} &= \sum_{\substack{a \in A \\ N/(1.9 v e^{-1/u}) < a \leq \lambda N/(8u)}} \frac{1}{a} \\
& \geq \sum_{\substack{a \in A \\ N/v < a \leq N/u}} \frac{1}{a}  - \sum_{\substack{N/v < a \leq \max\{N/v, N/(1.9 v e^{-1/u})\}}} \frac{1}{a}  - \sum_{\substack{a \in A \\ \frac{\lambda}{8u} N \leq a \leq N/u}} \frac{1}{a} \\
& \geq \frac{1+\lambda}{u} + \min\{0, \log (1.9 e^{-1/u})\} - O(1/N) - 0.3 - 0.27 \cdot \frac{\lambda}{u} \\
&\geq \min\{\log 1.9, 1/u\} - 0.32 + 0.73 \cdot \frac{\lambda}{u} \geq 0.73 \cdot \frac{\lambda}{u} \geq  \frac{1.4}{u'}
\end{split}
\]
where we used $u \leq 2.1$ and $u' \geq 2u/\lambda$ in the last two steps. Note also that $N' = Nv'/(1.9 v e^{-1/u}) \geq (100v/\lambda)^2 v'/(2v) \geq (30v'/\lambda)^2$ since $v \geq v'$. Hence we can now apply the first part of Hypothesis A with $N', u'$ and $v'$, to obtain $k_1 \leq v'$ such that
\[
|\{(a_1,\dots,a_{k_1})\in A^{k_1} \colon N-d-k_1 \leq a_1+\dotsb+a_{k_1} \leq N-d\}| \ge \alpha_{v'}' \frac{|A|^{{k_1}}}{N},
\]
where $\alpha_{v'}' > 0$ is a constant with $\alpha_v' = v^{-o(v)}$ as $v \rightarrow \infty$. Write $k = k_1 + 1 \leq v' + 1 \leq 0.98e^{-1/u}v + 1 \leq e^{-1/u}v$. We then get
\[
\begin{split}
&|\{(a_1,\dots,a_k)\in A^k \colon N-k \leq a_1+\dotsb+a_k \leq N\}|\\
& \ge \sum_{d \in D}  |\{(a_1,\dots,a_{k_1})\in A^{k_1} \colon N-d-k_1 \leq a_1+\dotsb+a_{k_1} \leq N-d\}| \\
&\geq \frac{\lambda^2}{100u^2}N \cdot  \alpha'_{v'} \frac{|A|^{k_1}}{N} \geq \alpha_v \frac{|A|^{k}}{N},
\end{split}
\]
where $\alpha_v = \frac{\lambda^2}{100u^2} \cdot \alpha_{v'}' = v^{-o_{\lambda}(v)}$ as $v \rightarrow \infty$.

\bibliographystyle{plain}
\bibliography{sieve-arxiv}{}

\begin{thebibliography}{10}

\bibitem{Bleichenbacher}
D.~Bleichenbacher.
\newblock The continuous postage stamp problem.
\newblock {\em Unpublished manuscript}, 2003.

\bibitem{Bo97}
J.~Bourgain.
\newblock Estimates related to sumfree subsets of sets of integers.
\newblock {\em Israel J. Math.}, 97:71--92, 1997.

\bibitem{Cassels75}
J.~W.~S. Cassels.
\newblock Measures of the non-convexity of sets and the
  {S}hapley-{F}olkman-{S}tarr theorem.
\newblock {\em Math. Proc. Cambridge Philos. Soc.}, 78(3):433--436, 1975.

\bibitem{EkTe76}
I.~Ekeland and R.~Temam.
\newblock {\em Convex analysis and variational problems}.
\newblock North-Holland Publishing Co., Amsterdam-Oxford; American Elsevier
  Publishing Co., Inc., New York, 1976.
\newblock Translated from the French, Studies in Mathematics and its
  Applications, Vol. 1.

\bibitem{Opera}
J.~Friedlander and H.~Iwaniec.
\newblock {\em Opera de cribro}, volume~57 of {\em American Mathematical
  Society Colloquium Publications}.
\newblock American Mathematical Society, Providence, RI, 2010.

\bibitem{GrKoMa}
A.~Granville, D.~Koukoulopoulos, and K.~Matom\"aki.
\newblock When the sieve works.
\newblock {\em Duke Math. J.}, 164:1935--1969, 2015.

\bibitem{Gr05}
B.~Green.
\newblock A {S}zemer\'edi-type regularity lemma in abelian groups, with
  applications.
\newblock {\em Geom. Funct. Anal.}, 15(2):340--376, 2005.

\bibitem{IwKo04}
H.~Iwaniec and E.~Kowalski.
\newblock {\em Analytic number theory}, volume~53 of {\em American Mathematical
  Society Colloquium Publications}.
\newblock American Mathematical Society, Providence, Rhode Island, 2004.

\bibitem{LePo}
H.~W.~Lenstra jr. and C.~Pomerance.
\newblock Primality testing with gaussian periods.
\newblock {\em Preprint, available at
  \verb+http://www.math.dartmouth.edu/~carlp/aks06-2015.pdf+}, 2011.

\bibitem{KSV09}
D.~Kr{\'a}l, O.~Serra, and L.~Vena.
\newblock A combinatorial proof of the removal lemma for groups.
\newblock {\em J. Combin. Theory Ser. A}, 116(4):971--978, 2009.

\bibitem{Ruz97}
I.~Z. Ruzsa.
\newblock The {B}runn-{M}inkowski inequality and nonconvex sets.
\newblock {\em Geom. Dedicata}, 67(3):337--348, 1997.

\bibitem{Sh10}
A.~Shapira.
\newblock A proof of {G}reen's conjecture regarding the removal properties of
  sets of linear equations.
\newblock {\em J. Lond. Math. Soc. (2)}, 81(2):355--373, 2010.

\bibitem{Sz10}
B.~Szegedy.
\newblock The symmetry preserving removal lemma.
\newblock {\em Proc. Amer. Math. Soc.}, 138(2):405--408, 2010.

\bibitem{SV06}
E.~Szemer{\'e}di and V.~H. Vu.
\newblock Finite and infinite arithmetic progressions in sumsets.
\newblock {\em Ann. of Math. (2)}, 163(1):1--35, 2006.

\bibitem{TV08}
T.~Tao and V.~Vu.
\newblock John-type theorems for generalized arithmetic progressions and
  iterated sumsets.
\newblock {\em Adv. Math.}, 219(2):428--449, 2008.

\bibitem{TV10}
T.~Tao and V.~H. Vu.
\newblock {\em Additive combinatorics}, volume 105 of {\em Cambridge Studies in
  Advanced Mathematics}.
\newblock Cambridge University Press, Cambridge, paperback edition, 2010.

\bibitem{Wolf10}
J.~Wolf.
\newblock The structure of popular difference sets.
\newblock {\em Israel J. Math.}, 179:253--278, 2010.

\end{thebibliography}

\end{document}